\documentclass{amsart}
\usepackage{latexsym, amsmath, amsthm, amsfonts, amssymb, txfonts, pxfonts, wasysym}
\usepackage[active]{srcltx}
\usepackage[latin1]{inputenc}
\usepackage{graphicx}
\usepackage{latexsym}
\usepackage{algorithm}
\usepackage{caption}
\usepackage[dvips]{color}
\usepackage{xy}
\usepackage{wrapfig}
\usepackage{dsfont}
\input xy
\xyoption{all}

\theoremstyle{plain}

\newtheorem{theorem}{Theorem}[section]
\newtheorem{lemma}[theorem]{Lemma}
\newtheorem{corollary}[theorem]{Corollary}
\newtheorem{proposition}[theorem]{Proposition}

\theoremstyle{definition}
\newtheorem{definition}[theorem]{Definition}

\theoremstyle{remark}
\newtheorem{remark}[theorem]{Remark}

\title{The collapsibility of CAT(0) square 2-complexes}

\author{Ioana-Claudia Laz\u{a}r}
\address{Politehnica University of Timi\c{s}oara,\\
Dept. of Mathematics,\\
Victoriei Square $2$, $300006$, Timi\c{s}oara,\\
Rom\^{a}nia}
\email{ioana.lazar@upt.ro}

\keywords{square $2$-complex, CAT(0) metric, elementary collapse, geodesic segment}

\subjclass{05C99, 05C75}

\thanks{The author was partially supported by the grant $346300$ for IMPAN from
the Simons Foundation and the matching $2015-2019$ Polish MNiSW fund. The author
was also partially supported by the grant $19/6-020/961-120/14$ of the Ministry for Science of the Republic of Srpska.}

\thanks{I thank Professor Tudor B\^{i}nzar for the careful reading of the paper and the useful remarks. I thank the anonymous referees for the helpful comments.}

\begin{document}

\begin{abstract}
  We give a sufficient condition for the collapsibility of finite square $2$-complexes.
  We show that any finite, CAT(0) square $2$-complex retracts to a point through CAT(0) subspaces.
\end{abstract}

\maketitle

\section{Introduction}

In this paper we investigate metric conditions which guarantee the
collapsibility of a finite, $2$-dimensional square complex.

The metric curvature condition we have in mind is given by the CAT(0)
inequality. A geodesic metric space is a CAT(0) space if geodesic
triangles are thinner than comparison triangles in the Euclidean space
(see \cite{bridson_1999,burago_2001,alex_1955,menger_1928}).
A $2$-dimensional polyhedral space is a space of nonpositive
curvature if and only if the link of each vertex does not contain a
subspace isometric to a circle of length less than $2 \pi$ (see
\cite{burago_2001}, chapter $4.2$, page $113$). Hence the standard
piecewise Euclidean metric structure on a $2$-dimensional simplicial
complex is nonpositively curved if and only if the link of each
vertex of the complex has girth at least $6$. The \textit{girth} of
a graph is defined as the minimum number of edges in a circuit.

Combinatorially, one can express curvature
using a condition, called $k$-systolicity ($k \geq 6$) which was introduced independently by Chepoi \cite{chepoi_2000} (under the name of bridged complexes), Januszkiewicz-{\' S}wi{\c a}tkowski \cite{JS1} and Haglund \cite{Hag}.
The theory of $7$-systolic groups, that is groups acting geometrically on $7$-systolic complexes, allowed to provide examples of highly dimensional Gromov hyperbolic groups (see \cite{JS0,JS1,O-chcg}). However, for groups acting geometrically on CAT(-1) cubical complexes or on $7$-systolic complexes, some very restrictive limitations
are known. For example, $7$-systolic groups are in a sense `asymptotically hereditarily aspherical', i.e.\ asymptotically they can not contain
essential spheres. This yields in particular that such groups are not fundamental groups of negatively curved manifolds of dimension
above two; see e.g. \cite{JS2,O-conn,O-ib,OS,O-ns}.
In \cite{O-sdn, chepoi_2009,BCCGO,ChaCHO} other conditions
of this type are studied. They form a way of unifying CAT(0) cubical and systolic theories.
Osajda introduced in \cite{O-8loc} another local combinatorial condition called \emph{$8$-location}. He showed that under the additional hypothesis
of local $5$-largeness, this condition implies Gromov hyperbolicity of the universal cover.
In \cite{L-8loc} we study of a version of $8$-location, suggested in \cite[Subsection 5.1]{O-8loc}.
This $8$-location says that homotopically trivial loops of length at most $8$ admit filling diagrams with one internal vertex.
In \cite{lazar_2015} we introduce another combinatorial curvature condition, called the $5/9$-condition, and we show that the complexes which fulfill it, are also Gromov hyperbolic.

The collapsibility of finite simplicial complexes was studied
before. White showed that a finite, strongly convex
$2$-complex, is collapsible (see \cite{white_1967}).
Corson-Trace proved further that a finite, simply
connected, $2$-dimensional simplicial complex that has the
$6$-property, collapses to a point (see \cite{corson_1998}).
In dimension above $2$, systolic simplicial complexes are also collapsible (see \cite{chepoi_2009,lazar_2012,lazar_2013}).
Crowley showed in \cite{crowley_2008} that a finite simplicial
complex of dimension $3$ or less endowed with the standard piecewise
Euclidean metric that is nonpositively curved, and satisfies a
technical condition, simplicially collapses to a point. She
constructed a CAT(0) $2$-complex by endowing the complex with
the corresponding piecewise Euclidean metric and requiring that each interior
vertex of the complex has degree at least $6$. The naturally
associated piecewise Euclidean metric on the $2$-complex becomes
then CAT(0). Crowley's result was extended by Adiprasito-Benedetti to all dimensions (see \cite{benedetti_2019}).
Using discrete Morse theory (see \cite{forman_1998,forman_2002}), they proved that every simplicial complex that is CAT(0) with a metric for which all vertex stars are convex, is collapsible.

It turns out that no
combinatorial condition is necessary to prove that a finite, CAT(0) simplicial
$2$-complex, is collapsible. A proof of this fact is given in \cite{lazar_2010_8} (chapter $3.1$, page $35$).
The aim of the present paper is to extend this result on square $2$-complexes.
In \cite{benedetti_2019} (Corollary $3.2.9$) it is shown that
every CAT(0) cube complex is collapsible. We give an alternative proof of the same result only in the $2$-dimensional case.

The main result of the paper states that a finite, CAT(0) square
$2$-complex retracts to a point through subspaces which are, at each
step of the retraction, CAT(0) spaces. When finding the new geodesic segments in the subspace obtained by performing an elementary collapse on a finite, CAT(0) square $2$-complex, we reduce the problem to the simplicial case. The proof for the fact that the subcomplex obtained by performing an elementary collapse on a finite CAT(0) square $2$-complex, remains non-positively curved, on the other hand, does not reduce to the simplicial case. Instead we argue on a square $2$-complex which is the novelty of the paper.

 In the infinite case the equivalent notion of collapsibility is called arborescent structure.
In \cite{corson_1998} it is shown that any locally finite, simply connected simplicial $2$-complex with the
$6$-property is a monotone union of a sequence of collapsible subcomplexes. We show that a similar result holds for
locally finite CAT(0) square $2$-complexes. This is a consequence of the fact that finite, CAT(0) square $2$-complexes are collapsible.

\section{Preliminaries}

We present in this section the notions we shall work with and the
results we shall refer to.

\begin{definition}

Let $(X,d)$ be a metric space. If $x,m,y$ are three points in $X$
such that $d(x,m) + d(m,y) = d(x,y)$, then we say that $m$
\emph{lies between} $x$ and $y$. We call $m$ the \emph{midpoint} of
$x$ and $y$ if $d(x,m) = d(m,y) = \frac{\textstyle 1}{\textstyle
2}d(x,y)$.

\end{definition}

\begin{definition}

Let $(X,d)$ be a metric space. $X$ is a \emph{convex} metric space
if for any two points $x,y$ in $X$, there exists at least one
midpoint $m$. $X$ is a \emph{strongly convex} metric space if for
any two points $x,y$ in $X$, there exists exactly one midpoint $m$.

\end{definition}

\begin{definition}
Let $(X,d)$ be a metric space and let $c : [a,b] \rightarrow X$ be a
path in $X$. The \emph{length} $l(c)$ of $c$ is defined by:
\begin{center}
$l(c) =$ $\underset{a = t_{0} \leq t_{1} \leq ... \leq t_{n} = b}
\sup$ $ \sum_{i=0}^{n-1}d(c(t_{i}), c(t_{i+1}))$,
\end{center}
where the supremum is taken over all possible partitions with $a =
t_{0} \leq t_{1} \leq ... \leq t_{n} = b$.
\end{definition}

\begin{definition}

Let $(X,d)$ be a metric space and let $x,y$ be two distinct points
in $X$. A \emph{segment} $c : [a,b] \rightarrow X$ in $X$ connecting
$x$ to $y$ is a path which has, among all paths joining $x$ to $y$ in
$X$, the shortest length.

\end{definition}

\begin{theorem}
Let $(X,d)$ be a metric space. Let $x$ and $y$ be two distinct
points in $X$.
\begin{enumerate}
\item A subset $S$ of $X$ containing $x$ and $y$
is a segment joining $x$ to $y$ if there exists a closed real line
interval $[a,b]$ and an isometry $c : [a,b] \rightarrow X$ such that
$c(a) = x$ and $c(b) = y$;
\item A path $c : [a,b]  \rightarrow X$
joining $x$ to $y$ is a segment from $x$ to $y$ if and only if $l(c)
= d(x,y)$.
\end{enumerate}
\end{theorem}

For the proof see \cite{alex_1955}, chapter II.2, page 76.

\begin{theorem}
Let $(X,d)$ be a complete metric space. There exists a segment (there exists a unique segment) in
$X$ between any two distinct
points $x , y$ in $X$ if and only if $X$ is a convex metric space (strongly convex metric space).
\end{theorem}

For the proof see \cite{menger_1928}.

\begin{definition}

Let $(X,d)$ be a metric space. A \emph{geodesic path} joining $x \in
X$ to $y \in X$ is a path $c : [a,b] \to X$ such that $c(a) = x$,
$c(b) = y$ and $d(c(t), c(t')) = |t - t'|$ for all $t, t' \in
[a,b]$. The image $\alpha$ of $c$ is called a \emph{geodesic
segment} with endpoints $x$ and $y$.

\end{definition}

A \emph{geodesic metric space} $(X,d)$ is a metric space in which
every pair of points can be joined by a geodesic segment. We denote
any geodesic segment from a point $x$ to a point $y$ in $X$, by
$[x,y]$. We emphasize that any such geodesic segment is not
determined by its endpoints. Thus, without further assumptions on
$X$, there may be many geodesic segments joining $x$ to $y$.

A \emph{geodesic triangle} in $X$ consists of three points $p,q,r
\in X$, called \emph{vertices}, and a choice of three geodesic
segments $[p,q], [q,r], [r,p]$ joining them, called \emph{sides}.
Such a geodesic triangle is denoted by $\triangle (p,q,r)$. If a point $x \in X$ lies in the
union of $[p,q], [q,r]$ and $[r,p]$, then we write $x \in
\triangle$. A triangle $\overline{\triangle} =
\triangle(\overline{p},\overline{p},\overline{r})$ in
$\mathds{R}^{2}$ is called a \emph{comparison triangle} for
$\triangle = \triangle (p,q,r)$ if $d(p,q) =
d_{\mathds{R}^{2}}(\overline{p},\overline{q})$, $d(q,r) =
d_{\mathds{R}^{2}}(\overline{q},\overline{r})$ and $d(r,p) =
d_{\mathds{R}^{2}}(\overline{r},\overline{p})$. A point
$\overline{x} \in [\overline{q},\overline{r}]$ is called a
\emph{comparison point} for $x \in [q,r]$ if $d(q,x) =
d_{\mathds{R}^{2}}(\overline{q},\overline{x})$. The interior angle
of $\overline{\triangle} =
\triangle(\overline{p},\overline{p},\overline{r})$ at $\overline{p}$
is called the \emph{comparison angle} between $q$ and $r$ at $p$ and
it is denoted by $\overline{\angle}_{p} (q,r)$ (the comparison angle
is well-defined provided $q$ and $r$ are both distinct from $p$).

\begin{definition}
Let $(X,d)$ be a metric space and let $c : [0,a] \rightarrow X$ and
$c' : [0,a'] \rightarrow X$ be two geodesic paths with $c(0) =
c'(0)$. Given $t \in (0,a]$ and $t' \in (0,a']$, we consider the
comparison triangle $\overline{\triangle} (c(0),c(t),c'(t'))$ in
$\mathds{R}^{2}$ and the comparison angle $\overline{\angle}_{c(0)}
(c(t),c'(t'))$. The \emph{Alexandrov angle} between the geodesic paths $c$ and $c'$ is the number
$\angle(c,c') \in [0, \pi]$ defined by:
\begin{center}
$\angle(c,c') := \underset{t,t' \rightarrow 0} \limsup
\overline{\angle}_{c(0)} (c(t),c'(t')) = \underset{\varepsilon
\rightarrow 0} \lim \underset{ 0 < t,t' < \varepsilon} \sup
\overline{\angle}_{c(0)} (c(t),c'(t'))$.
\end{center}
\end{definition}

The Alexandrov angle between two geodesic segments which have a
common endpoint, is defined to be the Alexandrov angle between the
unique geodesics which issue from this point and whose images are
the given segments. Alexandrov angles in $\mathds{R}^{2}$ are equal
to the usual Euclidean angles.

\begin{remark}
The Alexandrov angle between the geodesic paths $c : [0,a]
\rightarrow X$ and $c' : [0,a'] \rightarrow X$ in a metric space
$(X,d)$ depends only on the germs of these paths at $0$. If $c'' :
[0,a''] \rightarrow X$ is any geodesic path for which there exists
$\varepsilon > 0$ such that $c''\mid_{[0,\epsilon]} =
c'\mid_{[0,\epsilon]}$, then the Alexandrov angle between $c$ and
$c''$ is the same as that between $c$ and $c'$.
\end{remark}

\begin{definition}
Let $(X,d)$ be a convex metric space. Let $c : [0,a] \rightarrow X$
and $c' : [0,a'] \rightarrow X$ be two geodesic paths with $c(0) =
c'(0) = p$ which have no other common points in the neighborhood of
$p$. The geodesic paths $c$ and $c'$ divide a sufficiently small
neighborhood of $p$ into two \emph{sectors} $U$ and $V$. We consider
in $U$ the geodesic paths $c_{1}, c_{2}, ..., c_{n}$, numbered
according to their position relative to $c$ and $c'$. We denote by
$\alpha_{0}, \alpha_{1}, ..., \alpha_{n}$ the Alexandrov angles
between $c$ and $c_{1}$, $c_{1}$ and $c_{2}$, ..., $c_{n}$  and
$c'$, respectively. The upper limit of the sum $\alpha_{0} + \alpha_{1} + ... +
\alpha_{n}$ for any geodesic paths $c_{i}$ in $U$, $1 \leq i \leq
n$, is called the \emph{Alexandrov angle of the sector} $U$.
\end{definition}

\begin{definition}
Let $(X,d)$ be a convex metric space and let $p$ be a point in $X$.
Let $U_{1}, ..., U_{n}$ be sectors around $p$ which form a full
neighborhood of $p$. We call the sum of the Alexandrov angles of the
sectors $U_{1}, ..., U_{n}$ in $X$, \emph{the full angle around the
point} $p$ in $X$.
\end{definition}

\begin{definition}
Let $(X,d)$ be a convex metric space. Let $\triangle (p,q,r)$ be a
geodesic triangle in $X$. Let $\alpha, \beta, \gamma$ be the
Alexandrov angles between the sides of $\triangle$. The
\emph{curvature of the geodesic triangle} $\triangle$ is defined by
$\omega(\triangle) = \alpha + \beta + \gamma - \pi$.
\end{definition}

\begin{definition}
Let $(X,d)$ be a convex metric space. Let $p$ be a point of $X$. Let
$\theta$ be the full angle around the point $p$. The
\emph{curvature at the point} $p$ is defined by $\omega(p) = 2\pi -
\theta$.
\end{definition}

\begin{theorem}

Let $(X,d)$ be a convex metric space and let $\triangle (p,q,r)$ be
a geodesic triangle in $X$ whose curvature equals zero. Then
$\triangle (p,q,r)$ is isometric to its comparison triangle
$\overline{\triangle} (\overline{p},\overline{q},\overline{r})$ in
$\mathds{R}^{2}$.

\end{theorem}

For the proof we refer to \cite{alex_1955}, chapter V.6, page 218.

We define next CAT(0) spaces and present some of their basic
properties.

\begin{definition}
Let $(X,d)$ be a metric space. Let $\triangle (p,q,r)$ be a geodesic
triangle in $X$. Let $\overline{\triangle}
(\overline{p},\overline{q},\overline{r}) \subset \mathds{R}^{2}$ be
a comparison triangle for $\triangle$. The metric $d$ is
\emph{CAT(0)} if for all $x,y \in \triangle$ and all comparison
points $\overline{x}, \overline{y} \in \overline{\triangle}$, the
CAT(0) inequality holds:
$d(x,y) \leq d_{\mathds{R}^{2}}(\overline{x}, \overline{y})$.
A metric space $X$ is called a \emph{CAT(0) space} if it is a geodesic
space all of whose geodesic triangles satisfy the CAT(0) inequality.
A metric space $X$ is said to be \emph{of curvature $\leq 0$} (or
\emph{non-positively curved}) if it is locally a CAT(0) space, i.e.
for every $x \in X$ there exists $r_{x} > 0$ such that the ball
$B(x, r_{x})$, endowed with the induced metric, is a CAT(0) space.

\end{definition}

\begin{theorem}
Let $X$ be a CAT(0) space.
\begin{enumerate}
\item The balls in $X$ are convex (i.e., any
two points in such a ball are joined by a unique geodesic segment
and this segment is contained in the ball) and contractible;
\item (Approximate midpoints are close to midpoints.) For every $\varepsilon > 0$ there exists $\delta = \delta(\varepsilon) > 0$
such that if $m$ is the midpoint of a geodesic segment $[x,y]
\subset X$ and if
$\max \{ d(x,m'), d(y,m') \} \leq\frac{\textstyle 1}{\textstyle
2}d(x,y) + \delta$,
then $d(m,m') < \varepsilon$.
\end{enumerate}
\end{theorem}

For the proof we refer to \cite{bridson_1999}, chapter II.1, page 160.

In CAT(0) spaces, angles exist in the
following strong sense.

\begin{theorem}
Let $X$ be a CAT(0) space and let $c : [0,a] \rightarrow X$ and $c'
: [0,a'] \rightarrow X$ be two geodesic paths issuing from the same
point $c(0) = c'(0)$. Given $t \in (0,a]$ and $t' \in (0,a']$, let
$\overline{\triangle}(c(t),c(0),c'(t'))$ be a comparison triangle in
$\mathds{R}^{2}$ for $\triangle(c(t),c(0),c'(t'))$.
The comparison angle $\overline{\angle}_{c(0)}(c(t),c'(t'))$ is a non-decreasing
function of both $t,t' \geq 0$ and the Alexandrov angle
$\angle(c,c')$ is equal to $\lim _{t,t' \rightarrow 0}
\overline{\angle}_{c(0)} (c(t),c'(t')) = \lim _{t \rightarrow
0}\overline{\angle}_{c(0)} (c(t),c'(t))$. Hence \begin{center}
$\angle(c,c') = \underset{t \rightarrow 0} \lim $ \emph{2} $\arcsin
\frac{1}{2t} d(c(t),c'(t))$ \end{center}
\end{theorem}

For the proof see \cite{bridson_1999}, chapter II.3, page 184.

Let $p,x,y$ be points of a metric space $X$ such that $p \neq x$ and
$p \neq y$. If there are unique geodesic segments $[p,x]$ and
$[p,y]$, then we write $\angle_{p}(x,y)$ to denote the Alexandrov
angle between these segments.

\begin{theorem}
Let $X$ be a metric space. The following conditions are equivalent:
\begin{enumerate}
\item $X$ is a CAT(0) space;
\item for every geodesic triangle $\triangle (p,q,r)$ in $X$ and for
every point $x \in [q,r]$, the following inequality is satisfied by
the comparison point $\overline{x} \in [\overline{q}, \overline{r}]
\subset \overline{\triangle}(p,q,r) \subset \mathds{R}^{2}$:
$d(p,x) \leq d(\overline{p}, \overline{x})$;
\item the Alexandrov angle between the sides of any geodesic
triangle in $X$ with distinct vertices is not greater than the angle
between the corresponding sides of its comparison triangle in
$\mathds{R}^{2}$.
\end{enumerate}
\end{theorem}

For the proof see \cite{bridson_1999}, chapter II.1, page 161.

\begin{theorem}

Any CAT(0) space is contractible; in particular it is simply
connected.

\end{theorem}

For the proof we refer to \cite{bridson_1999}, chapter II.1, page 161.

\begin{theorem}
Let $(X,d)$ be a CAT(0) space. Then the distance function $d : X
\times X \rightarrow \mathds{R}$ is convex and strongly convex.
\end{theorem}

For the proof see \cite{bridson_1999}, chapter II.2, page 176 and chapter II.1, page 160.

\begin{theorem}
Let $X$ be a complete connected metric space. If $X$ is simply
connected and of curvature $\leq 0$, then $X$ is a CAT(0) space.
\end{theorem}

For the proof we refer to \cite{bridson_1999}, chapter II.4, page 194.

Alexandrov's Lemma, given below, will be referred to frequently when
showing the main result of the paper.

\begin{lemma}\label{2.47}
Consider four distinct points $A, B, B', C$ in the Euclidean plane.
Suppose that $B$ and $B'$ lie on opposite sides of the line through
$A$ and $C$.
Consider the geodesic triangles $\triangle = \triangle(A, B, C)$ and
$\triangle' = \triangle(A, B', C)$. Let $\alpha, \beta, \gamma$
($\alpha', \beta', \gamma'$) be the angles of $\triangle$
($\triangle'$) at the vertices $A, B, C$ ($A, B', C$).

Let $\overline{\triangle}$ be a triangle in $\mathds{R}^{2}$ with
vertices $\overline{A}, \overline{B}, \overline{B}'$ such that $
d(\overline{A}, \overline{B}) = d(A, B)$, $d(\overline{A},
\overline{B}') = d(A, B')$ and $d(\overline{B}, \overline{B}') =
d(B, C) + d(C, B')$. Let $\overline{C}$ be the point of
$[\overline{B}, \overline{B}']$ with $d(\overline{B}, \overline{C})
= d(B, C)$. Let $\overline{\alpha}, \overline{\beta},
\overline{\beta}'$ be the angles of $\overline{\triangle}$ at the
vertices $\overline{A}, \overline{B}, \overline{B}'$.

If $\gamma + \gamma' \geq \pi$ then,
$d(B,C) + d(B', C) \leq d(B,A) + d(B', A)$. Also $\overline{\alpha} \geq \alpha + \alpha', \overline{\beta} \geq
\beta,  \overline{\beta}' \geq \beta'$ and $d(\overline{A},
\overline{C}) \geq d(A, C)$.

If $\gamma + \gamma' \leq \pi$ then,
$d(B,C) + d(B', C) \geq d(B,A) + d(B', A)$.
Also, $\overline{\alpha} \leq \alpha + \alpha', \overline{\beta} \leq
\beta,  \overline{\beta}' \leq \beta'$ and  $d(\overline{A},
\overline{C}) \leq d(A, C)$.

Any one equality is equivalent to the others, and occurs if and only
if $\gamma + \gamma' = \pi$.
\end{lemma}

For the proof see \cite{bridson_1999}, chapter I.2, page 25.

\begin{definition}
The unit $n$-cube $I^{n}$ is the $n$-fold product $[0,1]^{n}$; it is isometric to a cube in the Euclidean $n$-space
with edges of length one. By convention, $I^{0}$ is a point.
A \emph{cubical} \emph{complex}  $K$ is the quotient of a disjoint
union of cubes $X = \bigcup_{\Lambda}I^{n_{\lambda}}$ by an
equivalence relation $\sim$. The restrictions $p_{\lambda} :
I^{n_{\lambda}} \rightarrow K$ of the natural projection $ p : X
\rightarrow K = X| _{\sim}$ are required to satisfy:
\begin{enumerate}
\item for every $\lambda \in \Lambda$, the map $p_{\lambda}$ is
injective;
\item if $p_{\lambda}( I^{n_{\lambda}}) \bigcap p_{\lambda'}( I^{n_{\lambda'}}) \neq
\emptyset$ then there is an isometry $h_{\lambda, \lambda'}$ from a
face $T_{\lambda} \subset I^{n_{\lambda}}$ onto a face $T_{\lambda'}
\subset I^{n_{\lambda'}}$ such that $p_{\lambda}(x) =
p_{\lambda'}(x')$ if and only if $x' = h_{\lambda, \lambda'}(x)$.
\end{enumerate}

\end{definition}

In other words, $K$ is a cubical complex if and only if each of its
cells $C_{\lambda}$ is isometric to a cube $I^{n_{\lambda}}$, each
of the maps $p_{\lambda}$ is injective, and the intersection of any
two cells in $K$ is empty or a single face.
We call a $2$-dimensional cubical complex a \emph{square} complex.

\begin{definition}
 Let $K$ be a square complex. Let $\sigma$
be a $2$-cell of $K$ with vertices at the points $a,b,c,d$. The \emph{curvature}
of $\sigma$ is equal to $\omega(\sigma) = [\angle_{a}(b,d) + \angle_{b}(a,d) + \angle_{d}(a,b)] + [\angle_{c}(b,d) + \angle_{b}(c,d) + \angle_{d}(b,c)] - 2 \pi$.

\end{definition}

\begin{definition}
Let $K$ be a square $2$-complex and let $\alpha$ be an $i$-cell
of $K$, $0 \leq i \leq 2$. If $\beta$ is a $k$-dimensional face of $\alpha, k < i$ but not of
any other cell in $K$, then we say there is an \emph{elementary
collapse} from $K$ to $K \setminus \{\alpha, \beta\}$. If $K = K_{0}
\supseteq K_{1} \supseteq ... \supseteq K_{n} = L$ are square
complexes such that there is an elementary collapse from $K_{j-1}$
to $K_{j}$, $1 \leq j \leq n $, then we say that $K$
\emph{collapses} to $L$.
\end{definition}

\begin{definition}
A \emph{locally finite} space is a topological space in which every point has a finite neighborhood.
\end{definition}

\begin{definition}
Let $K$ be a locally finite square complex. We say $K$ has an {\it arborescent structure}
if it is a monotone union $\cup_{n=1}^{\infty}L_{n}$ of a sequence of collapsible subcomplexes $L_{n}$.
\end{definition}

\begin{definition}
Let $K$ be an $n$-dimensional square complex. A $k$-dimensional
subcomplex $K'$ of $K$ is called a \emph{spine} of $K$ if $K$
collapses to $K'$ for any $k < n$.
\end{definition}

\begin{definition}
A point $a$ of an $n$-dimensional square complex $K$ is an
\emph{interior point} of $K$ if it is contained in a subspace $U$ of $K$ which is homeomorphic to an $n$-ball $B^{n}$ of finite radius. Otherwise we call $a$ an \emph{exterior point} of
$K$.
\end{definition}

\begin{definition}
Let $K$ be a square complex and let $e$ be an edge of $K$.
We denote by
$i(e)$, the \emph{initial} vertex of $e$, by $t(e)$, the \emph{terminus} of $e$. A finite sequence $
e_{0}e_{2}...e_{n}$ of edges in $K$ such that
$t(e_{i}) = i(e_{i+1})$ for all $0 \leq i \leq n-1$, is called an
\emph{edge-path} in $K$.
\end{definition}

\begin{definition}
 A subcomplex $L$ of a square complex $K$ is called \emph{full} (in $K$) if any cell of $K$ spanned by a set of vertices in $L$, is a cell of $L$.
\end{definition}

\begin{definition}
We call a square complex $K$ \emph{flag}
if any set of vertices is included in a face of $K$ whenever each
pair of its vertices is contained in a face of $K$.
\end{definition}

\section{Collapsing a CAT(0) square 2-complex}

This section provides a metric characterization of collapsible square
$2$-com\-ple\-xes. We show that finite, CAT(0) square $2$-complexes are collapsible.
Namely, they retract to a point through CAT(0)
subspaces. Similar results are obtained in \cite{lazar_2010_8} on
finite CAT(0) simplicial $2$-complexes.

We start investigating the collapsibility of finite, CAT(0) square $2$-complexes by characterizing the curvature at the interior points of such complex.
Besides, in the following proposition we show that finite, CAT(0) square $2$-complexes have a $2$-cell with a free
$1$-dimensional face.

\begin{proposition}
Let $K$ be a finite square $2$-complex. If $|K|$
admits a CAT(0) metric $d$, then:

\begin{enumerate}

\item $|K|$ has curvature $\leq 0$ at any of its interior points;

\item $K$ has a $2$-cell with a free $1$-dimensional face.

\end{enumerate}

\end{proposition}

\begin{proof}

\begin{enumerate}

\item

Let $\tau = [c,h]$ be a $1$-cell of $K$ that is the face of at least two
$2$-cells in $K$, $\sigma_{1}$ and $\sigma_{2}$. Let $a,b,c,h$ be the vertices of the $2$-cell $\sigma_{1}$. Let $c,h,e,f$
be the vertices of the $2$-cell $\sigma_{2}$. Let $g \in \tau$.

\begin{figure}[ht]
   \begin{center}
     \includegraphics[height=2.5cm]{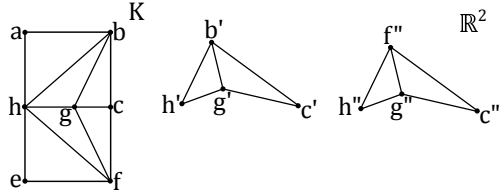}
        \caption{Comparison triangles}
 \end{center}
\end{figure}

Let $\triangle (b',h',g')$ be a comparison triangle for the geodesic triangle \\
$\triangle (b,h,g)$. Let $\triangle
(b',c',g')$ be a comparison triangle for the geodesic triangle
$\triangle (b,c,g)$. We place the comparison triangles $\triangle
(b',h',g')$ and $\triangle (b',c',g')$ in different half-planes with
respect to the line $b'g'$ in $\mathds{R}^{2}$. Let $\triangle
(h'',f'',g'')$ be a comparison triangle for the geodesic triangle
$\triangle (h,f,g)$. Let $\triangle (c'',f'',g'')$ be a
comparison triangle for the geodesic triangle $\triangle (c,f,g)$. We
place the comparison triangles $\triangle (h'',f'',g'')$ and
$\triangle (c'',f'',g'')$ in different half-planes with respect to
the line $f''g''$.

Because $d_{\mathds{R}^{2}}(b',g') = d(b,g)$ and
$d_{\mathds{R}^{2}}(f'',g'') = d(f,g)$, Alexandrov's Lemma implies
that $\angle_{g'} (h',b') +$ $\angle_{g'} (b',c')$ $= \pi$ and
$\angle_{g''} (h'',f'') + \angle_{g''} (f'',c'') = \pi$. So
$\angle_{g} (h,b) + \angle_{g} (b,c) = \pi$ and $\angle_{g} (h,f) +
\angle_{g} (f,c) = \pi$. Hence, because the $1$-cell $\tau$ is
contained in at least two $2$-cells of $K$, the full angle
around the point $g$ in $|K|$ equals at least $2 \pi$.

We denote by $\theta$
the full angle around the point $g$ in $|K|$. Because $|K|$ has a
convex metric, the curvature at the interior point $g$ of $|K|$ is equal to $\omega (g) = 2 \pi
- \theta \leq 0$. $|K|$ has therefore
curvature $\leq 0$ at any of its interior points.

\item Let $\tau = [c,h]$ be a $1$-cell of $K$ such that there exist at least two
$2$-cells $\sigma_{1}$ and $\sigma_{2}$ in $K$ with $\tau <
\sigma_{1}$ and $\tau < \sigma_{2}$. Let $a,b,c,h$ be the vertices of the
$2$-cell $\sigma_{1}$, and let $c,h,e,f$ be the vertices of the
$2$-cell $\sigma_{2}$. Let $g \in \tau$.

\begin{figure}[ht]
   \begin{center}
     \includegraphics[height=2.5cm]{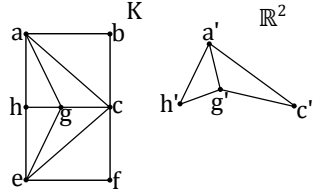}
        \caption{Comparison triangles}
 \end{center}
\end{figure}

Let $\triangle (a',h',g')$ be a comparison triangle for the geodesic triangle $\triangle (a,h,g)$ in $|K|$. Let $\triangle
(a',c',g')$ be a comparison triangle for the geodesic triangle
$\triangle (a,c,g)$ in $|K|$. We place the comparison triangles $\triangle
(a',h',g')$ and $\triangle (a',c',g')$ in different half-planes with
respect to the line $a'g'$.

Because any geodesic triangles in $|K|$ satisfies the CAT(0)
inequality and $g \in [h,c]$, $\pi = \angle_{g} (h,c) \leq
\angle_{g} (h,a) + \angle_{g} (a,c) \leq \angle_{g'} (h',a') +
\angle_{g'} (a',c')$. So $\angle_{g'} (h',a') + \angle_{g'} (a',c')
\geq \pi$. According to Alexandrov's Lemma, we have
$d_{\mathds{R}^{2}}(a',g') \leq d(a,g)$.
But $\triangle (a',h',g')$ is a comparison triangle for the geodesic triangle $\triangle
(a,h,g)$ and hence $d_{\mathds{R}^{2}}(a',g') = d(a,g)$. Because one
equality in Alexandrov's Lemma implies the others, the following
equalities hold $\angle_{g'} (h',a') + \angle_{g'} (a',c') = \pi$,
$\angle_{h} (a,g) = \angle_{h'} (a',g')$, $\angle_{c} (a,g) =
\angle_{c'} (a',g')$ and $\angle_{a} (h,g) + \angle_{a} (g,c) =
\angle_{a'} (h',c')$. Note that $\angle_{a'}(h',c') + \angle_{h'}(a',g') + \angle_{c'}(a',g') = \pi$. So the sum of the angles between the sides of
the geodesic triangle $\triangle(a,h,c)$ equals $\pi$. Therefore, since $|K|$ has a convex
metric, the curvature of the geodesic triangle $\triangle(a,h,c)$ equals
$\omega(\triangle(a,h,c)) = \pi - \pi = 0$.

 It similarly follows that the geodesic triangles $\triangle(a,b,c)$, $\triangle(b,h,c)$,
 and $\triangle(a,b,h)$ in $|K|$ have curvature zero.
So these triangles are isometric to their comparison triangles and are therefore flat.
Note that the geodesic triangles $\triangle(a,h,c)$, $\triangle(b,h,c)$,
$\triangle(a,b,c)$ and $\triangle(a,b,h)$ overlap and cover the $2$-cell $\sigma_{1}$.
So the curvature of $\sigma_{1}$ is equal to $\omega(\sigma_{1}) = [\angle_{a}(h,c) + \angle_{h}(a,c) + \angle_{c}(a,h)] +
[\angle_{a}(b,c) + \angle_{b}(a,c) + \angle_{c}(a,b)] - 2 \pi = \pi + \pi - 2\pi = 0$.
Thus $\sigma_{1}$ is flat.

Assume that $K$ has no
$2$-cell with a free $1$-dimensional face.
Hence each $1$-cell of $K$ is a face of at least two $2$-cells whose $1$-cells are further faces of at least two
$2$-cells and so on. $K$ is contractible because it is a CAT(0) space. Therefore, since $K$ is finite and the $2$-cells of $K$ are flat,
this implies a contradiction.
So $K$ has a $2$-cell with a free $1$-dimensional face.

\end{enumerate}

\end{proof}

Related to the result above we note the following.

\begin{remark} There are contractible finite $2$-complexes which are not collapsible. The dunce hat space, for instance, is a contractible surface without boundary (so no triangulation of it can have any free faces) (see \cite{zeeman_1964}). But the dunce hat is the flag triangulation of
a contractible $2$-complex which does not admit a CAT($0$) metric (see \cite{crisp_2002}).

\end{remark}

We show further that the subcomplex $K'$ obtained by performing an
elementary collapse on a finite, CAT(0) square $2$-complex $K$, remains
non-positively curved. We treat only the general case when $K'$ is
obtained by pushing in an entire $2$-cell with a free
$1$-dimensional face, by starting at its free face. We emphasize
that the same result holds for any deformation retract of a finite,
CAT(0) square $2$-complex $K$ obtained by pushing in any geodesic triangle
$\delta$ in $|K|$ that belongs to the $2$-cell of $K$ which has a free $1$-dimensional face. Namely, one side of $\delta$ is included in the free
$1$-dimensional face of this $2$-cell.

For the remainder of the paper we fix the following notaion.
Let $K$ be a finite, CAT(0) square $2$-complex such that it has a $2$-cell $\sigma$ which has a free $1$-dimensional
face $e$. Let $d$ be the CAT(0) metric $|K|$ is endowed with. We show that the subcomplex $K' = K \setminus \{e, \sigma \} $ is
non-positively curved.

Let $a,b,c,h$ be the vertices of the $2$-cell $\sigma$. Let $e =
[b,c]$ be its free $1$-dimensional face. We denote by $r := \max
\{d(a,b), $ $d(a,c), d(a,h) \}$. We consider in $|K|$ a neighborhood of $a$
homeomorphic to a closed ball of radius $r$, $U = \{x \in |K| \mid
d(a,x) \leq r\}$. $U$ endowed with the induced metric, is a CAT(0)
space. Because $U$ is complete and it has a strongly convex metric,
any two points in $U$ are joined by a unique geodesic segment which
is contained in $U$. So any geodesic triangle with vertices at any
three points in $U$, belongs to $U$ and it satisfies the CAT(0)
inequality.

We consider in $|K'|$ a neighborhood of $a$ homeomorphic to a closed
ball of radius $r$, $U' = \{x \in |K'| \mid d'(a,x) \leq r\}$,
endowed with the induced metric $d'$. We notice that $U' = U
\setminus \{e, \sigma\}$. We find next the geodesic segments in $U'$
with respect to $d'$.
Let $p$, $q$ be two distinct points in $U$ that do not belong to $\sigma$ such that the geodesic segment $[p,q]$
intersects the interior of $\sigma$. One of the points $p$ and $q$ may lie on one of the edges of $\sigma$ but not both. There are two cases to consider which are given below.

Case A: The segment $[p,q]$ intersects $[a,b]$ in $m$, and $[a,h]$ in $n$.
The segment $[p,q]$ does not intersect $[c,h]$ and $[b,c]$. The point $p$ may belong to $[a,b]$ and the point $q$ may belong to $[a,h]$ but not simultaneously.

\begin{figure}[ht]
   \begin{center}
     \includegraphics[height=2.8cm]{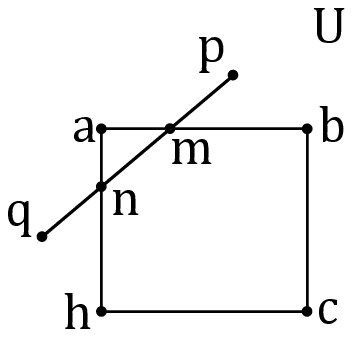}
        \caption{Case A}
 \end{center}
\end{figure}

Case B: The segment $[p,q]$ intersects $[a,b]$ in $m$, and $[c,h]$ in $n$.
The segment $[p,q]$ does not intersect $[a,h]$ and $[b,c]$. The point $p$ may belong to $[a,b]$ and the point $q$ may belong to $[c,h]$ but not simultaneously.

\begin{figure}[ht]
   \begin{center}
     \includegraphics[height=3.5cm]{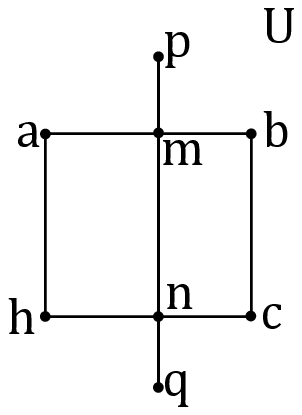}
        \caption{Case B}
 \end{center}
\end{figure}

The aim of the Lemmas \ref{3.7} - \ref{3.11} below is to find some geodesic segments in the subcomplex obtained by performing an elementary collapse on a finite, CAT(0) square $2$-complex.

\begin{lemma}\label{3.7}
Let $c : [0,1] \rightarrow U$ be a path in $U$ joining $p$ to $q$
that does not intersect $\sigma$. The points $p$ and $q$ are chosen as in Case A. Then there exists a point $s_{0}$
on $c$ such that the geodesic segments $[p,s_{0}]$ and $[q,s_{0}]$
do not intersect $\sigma$ and such that the following inequality
holds: $d'(p,a) + d'(a,q) < d'(p,s_{0}) + d'(s_{0},q)$.
\end{lemma}

\begin{proof}

We turn the square complex $K$ into a simplicial complex $L$. So $|K| = |L|$. Let $L$ be a finite simplicial $2$-complex such that all $i$-simplices of $K$ are also $i$-simplices of $L, 0 \leq i \leq 1$. Besides $L$ contains, for each $2$-cell $\sigma$ of $K$, the $1$-simplex $[a,h]$ and the $2$-simplices $\triangle(a,b,h)$, $\triangle(a,c,h)$. Note that $L$ is a CAT(0) space. The lemma now follows due to \cite{lazar_2010_8}, chapter $3.1.4$, page $39$.

\begin{figure}[ht]
   \begin{center}
     \includegraphics[height=2.5cm]{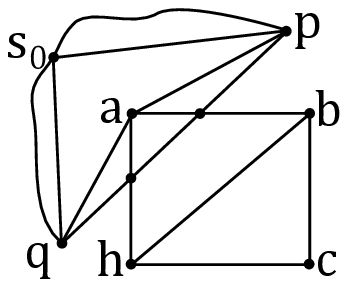}
        \caption{}
 \end{center}
\end{figure}

\end{proof}

\begin{lemma}\label{3.9}
Let the segment $[p,q]$ be as in case A.
Then the geodesic segment $[p,q]$ in $U'$ with respect to $d'$, is the
union of the geodesic segments $[p,a]$ and $[a,q]$.
\end{lemma}

\begin{proof}

We denote by $c : [0,1] \rightarrow U'$ the path obtained by
concatenating the segments $[p,a]$ and $[a,q]$. Among all paths
joining $p$ to $q$ in $U'$ that pass through $a$, the path $c$ has
the shortest length.

\begin{figure}[ht]
   \begin{center}
     \includegraphics[height=2.5cm]{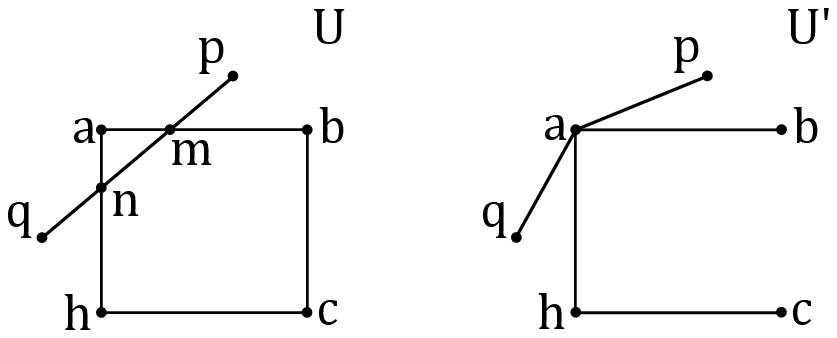}
        \caption{Case A}
 \end{center}
\end{figure}

Suppose that there exists a path $c_{0} : [0,1] \rightarrow U'$
connecting $p$ to $q$ in $U'$ that does not pass through $a$ and
whose length is less or equal to the length of the path $c$. Because
the path $c_{0}$ does not intersect $\sigma$, there exists,
according to Lemma \ref{3.7}, a point $s_{0}$ on $c_{0}$  such that
the geodesic segments $[p,s_{0}]$ and $[s_{0},q]$  in $U$ do not intersect
$\sigma$. The geodesic segments  $[p,s_{0}]$ and $[s_{0},q]$ in $U$
belong therefore to $U'$. So

$$d'(p,s_{0}) + d'(s_{0},q) \leq l(c_{0}) \leq l(c) = d'(p,a) + d'(a,q)$$

which implies, by Lemma \ref{3.7}, a contradiction. Any path in $U'$
joining $p$ to $q$ and that does not pass through $a$, is therefore longer
than $c$.

Altogether, it follows that the geodesic segment joining $p$ to $q$
in $U'$ with respect to $d'$ is the union of the geodesic segments
$[p,a]$ and $[a,q]$.

\end{proof}

\begin{lemma}\label{3.11}
Let the segment $[p,q]$ be as in case B.
Then the geodesic segment $[p,q]$ in $U'$ with respect to $d'$, is the
union of the geodesic segments $[p,a]$, $[a,h]$ and $[h,q]$.
\end{lemma}

\begin{proof}

Let $f$ be the midpoint of the geodesic segment $[a,h]$. Because $U$ is a CAT(0) space, the midpoint $f$ exists and it is unique.
Let $[p,f] \cap [a,b] = \{ m'\}$ and let $[q,f] \cap [c,d] = \{ n'\}$.
The previous lemma implies that the geodesic segment $[p,f]$ in $U'$ with respect to $d'$, is the
union of the geodesic segments $[p,a]$ and $[a,f]$.
It also implies that the geodesic segment $[q,f]$ in $U'$ with respect to $d'$, is the
union of the geodesic segments $[q,h]$ and $[h,f]$.
The geodesic segment $[p,q]$ in $U'$ with respect to $d'$, is therefore the
union of the geodesic segments $[p,a]$, $[a,f]$, $[f,h]$ and $[h,q]$. So the geodesic segment $[p,q]$ in $U'$ with respect to $d'$, is the
union of the geodesic segments $[p,a]$, $[a,h]$ and $[h,q]$.

\begin{figure}[ht]
   \begin{center}
     \includegraphics[height=3.5cm]{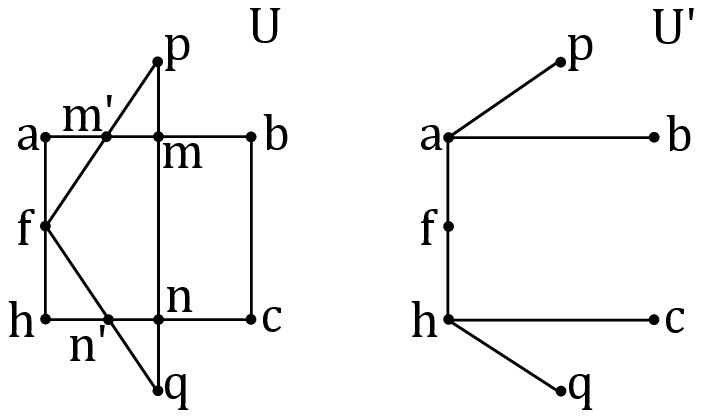}
        \caption{Case B}
 \end{center}
\end{figure}

\end{proof}

In the Lemmas \ref{3.13} - \ref{3.23} below we study whether certain geodesic triangles in $U'$
fulfill the CAT(0) inequality. This will be useful when showing that $U'$ is non-positively curved.

\begin{lemma}\label{3.13}
Let the segment $[p,q]$ be as in Case A.
Let $r$ be a point in $U$ such that the geodesic segments $[r,p]$
and $[r,q]$ do not intersect $\sigma$. Also the quadrilaterals $\rm{ramp}$ and $\rm{ranq}$ are convex. Then, the geodesic
triangle $\triangle (p,r,q)$ in $U'$ satisfies the CAT(0)
inequality.
\end{lemma}

\begin{proof}

By Lemma \ref{3.9}, $d'(p,q) = d'(p,a) + d'(a,q)$.

\begin{figure}[ht]
   \begin{center}
     \includegraphics[height=2.3cm]{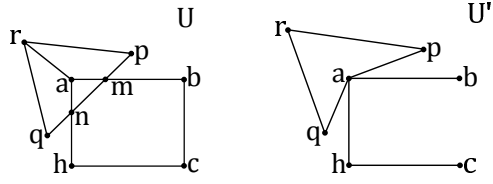}
        \caption{Case A, the quadrilaterals \emph{ramp} and \emph{ranq} are convex}
 \end{center}
\end{figure}

\begin{figure}[ht]
   \begin{center}
     \includegraphics[height=3.3cm]{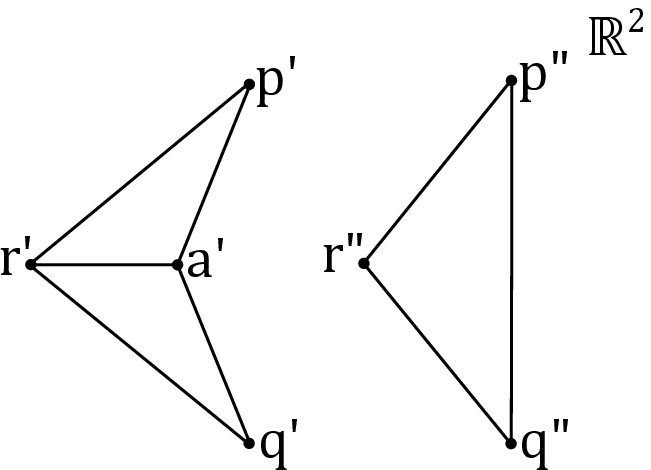}
        \caption{}
 \end{center}
\end{figure}

Let $\triangle(p',a',r')$ be a comparison triangle for $\triangle (p,a,r)$ in $U'$. Let $\triangle(r',a',q')$ be a comparison triangle for $\triangle (r,a,q)$ in $U'$. We place the comparison triangles $\triangle(p',a',r')$ and $\triangle(r',a',q')$ in different half-planes with respect to the line $a'r'$.
The CAT(0) inequality implies that $\angle_{p}(r,a) \leq \angle_{p'}(r',a')$, $\angle_{r}(p,a) \leq \angle_{r'}(p',a')$, $\angle_{r}(a,q) \leq \angle_{r'}(a',q')$, $\angle_{q}(r,a) \leq \angle_{q'}(r',a')$.

Let $\triangle(p'',q'',r'')$ be a comparison triangle for $\triangle (p,q,r)$ in $U'$.
Note that either $\angle_{a'}(p',r') + \angle_{a'}(r',q') \geq \pi$ or $\angle_{r'}(p',a') + \angle_{r'}(a',q') \geq \pi$. Assume w.l.o.g. that $\angle_{a'}(p',r') + \angle_{a'}(r',q') \geq \pi$. Alexandrov's Lemma implies that $\angle_{p'}(r',a') \leq \angle_{p''}(r'',q'')$,
$\angle_{r'}(p',a') +\angle_{r'}(a',q') \leq \angle_{r''}(p'',q'')$, $\angle_{q'}(r',a') \leq \angle_{q''}(r'',p'')$.

In conclusion $\angle_{p}(r,a) \leq \angle_{p''}(r'',q'')$, $\angle_{r}(p,q) \leq \angle_{r}(p,a) + \angle_{r}(a,q) \leq \angle_{r''}(p'',q'')$,
$\angle_{q}(r,a) \leq \angle_{q''}(r'',p'')$. So the geodesic triangle $\triangle (p,r,q)$ in $U'$ satisfies the CAT(0)
inequality.

\end{proof}

\begin{lemma}\label{3.15}
Let the segment $[p,q]$ be as in Case A.
Let $r$ be a point in $U$ such that the geodesic segments $[r,p]$
and $[r,q]$ do not intersect $\sigma$. The point $r$ is considered such that the quadrilateral $\rm{ramp}$ is convex and the quadrilateral $\rm{ranq}$ is concave. Then, the geodesic
triangle $\triangle (p,n,r)$ in $U'$ satisfies the CAT(0)
inequality.
\end{lemma}

\begin{proof}

By Lemma \ref{3.9}, $d'(p,n) = d'(p,a) + d'(a,n)$ and $d'(r,n) = d'(r,a) + d'(a,n)$.

\begin{figure}[ht]
   \begin{center}
     \includegraphics[height=2.3cm]{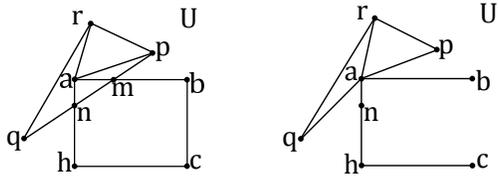}
        \caption{Case A, the quadrilateral  \emph{ramp} is convex and the quadrilateral \emph{ranq} is concave}
 \end{center}
\end{figure}

\begin{figure}[ht]
   \begin{center}
     \includegraphics[height=2.3cm]{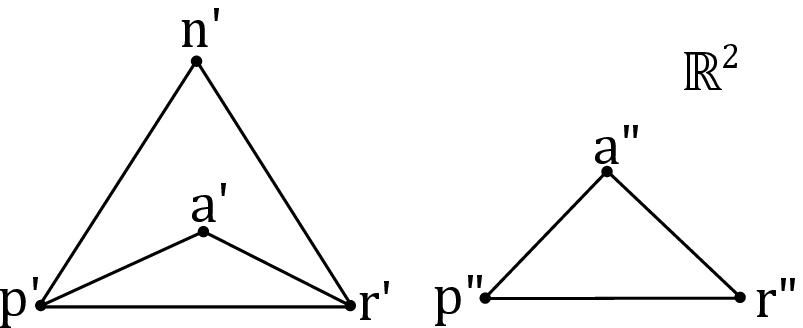}
        \caption{}
 \end{center}
\end{figure}

Let $\triangle(p'',a'',r'')$ be a comparison triangle for $\triangle (p,a,r)$ in $U'$. The CAT(0) inequality implies that $\angle_{p}(a,r) \leq \angle_{p''}(a'',r'')$,
$\angle_{r}(p,a) \leq \angle_{r''}(p'',a'')$.

Let $\triangle(p',n',r')$ be a comparison triangle for $\triangle (p,n,r)$ in $U'$. We consider a point $a'$ in the interior of the geodesic triangle
$\triangle (p',n',r')$ such that $d_{\mathds{R}^{2}}(p',a') =
d(p,a)$  and $d_{\mathds{R}^{2}}(r',a') = d(r,a)$. We can choose the point $a'$ in this manner because in $U'$ we have $a \in [p,n]$ and $a \in [r,n]$. Thus $d(p,a) < d(p,n)$ and $d(r,a) < d(r,n)$. Thus
$\angle_{p'} (a',r') < \angle_{p'} (n',r')$ and $\angle_{r'}
(a',p') < \angle_{r'} (n',p')$. Since the geodesic triangles
$\triangle (p',a',r')$ and $\triangle (p'',a'',r'')$ are congruent
to each other, $\angle_{p'} (a',r') \equiv \angle_{p''} (a'',r'')$
and $\angle_{r'} (a',p') \equiv \angle_{r''} (a'',p'')$.

In conclusion for the geodesic triangle $\triangle(p,n,r)$ in $U'$ we have $\angle_{p}(n,r) < \angle_{p'}(n',r')$, $\angle_{r}(n,p) < \angle_{r'}(n',p')$. Because the Alexandrov angle between the geodesic segments $[p,n]$ and $[r,n]$ equals zero, we have $\angle_{n}(r,p) = 0 < \angle_{n'}(p',r')$. Hence the geodesic
triangle $\triangle (p,n,r)$ in $U'$ satisfies the CAT(0)
inequality.

\end{proof}

\begin{lemma}\label{3.17}
Let the segment $[p,q]$ be as in case B. Let $r$ be a point in $U$ such that the geodesic segments  $[r,p]$
and $[r,q]$ do not intersect $\sigma$. Also $r$ is chosen such that either the
quadrilaterals $\rm{ramp}$ and $\rm{rhnq}$ are concave or the
quadrilateral $\rm{ramp}$ is convex and the quadrilateral $\rm{rhnq}$ are concave. Then, the geodesic
triangles $\triangle (p,r,q)$ and $\triangle (p,n,r)$ in $U'$ satisfy the CAT(0)
inequality.
\end{lemma}

\begin{proof}

By Lemma \ref{3.11}, $d'(p,q) = d'(p,a) + d'(a,h) + d'(h,q)$, $d'(p,n) = d'(p,a) + d'(a,h) + d'(h,n)$. By Lemma \ref{3.9},  $d'(r,n) = d'(r,h) + d'(h,n)$.

\begin{figure}[ht]
   \begin{center}
     \includegraphics[height=3cm]{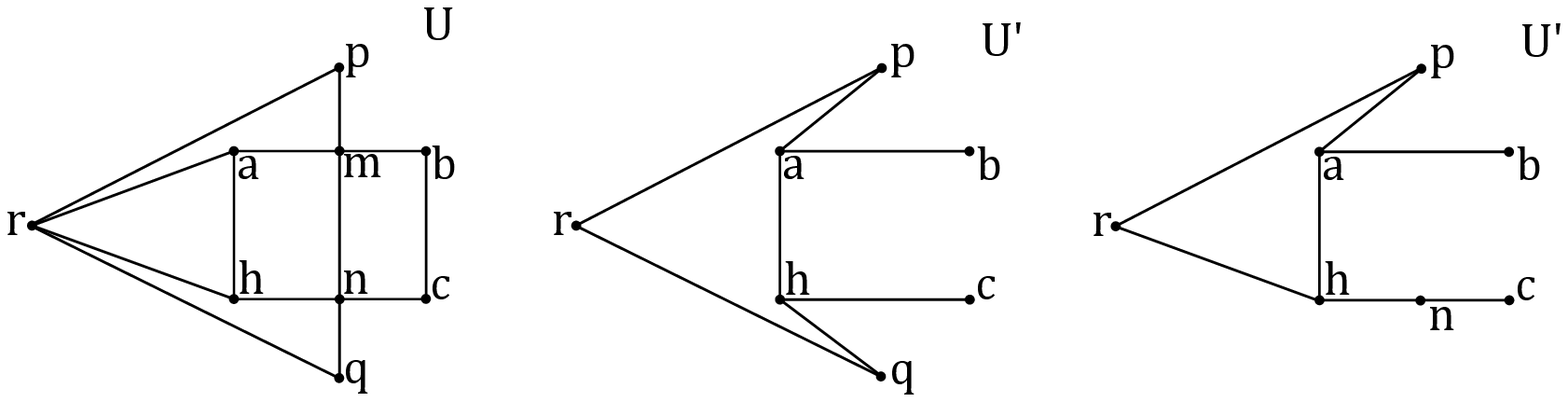}
        \caption{Case B, the quadrilaterals \emph{ramp} and \emph{rhnq} are concave}
 \end{center}
\end{figure}

\begin{figure}[ht]
   \begin{center}
     \includegraphics[height=3.5cm]{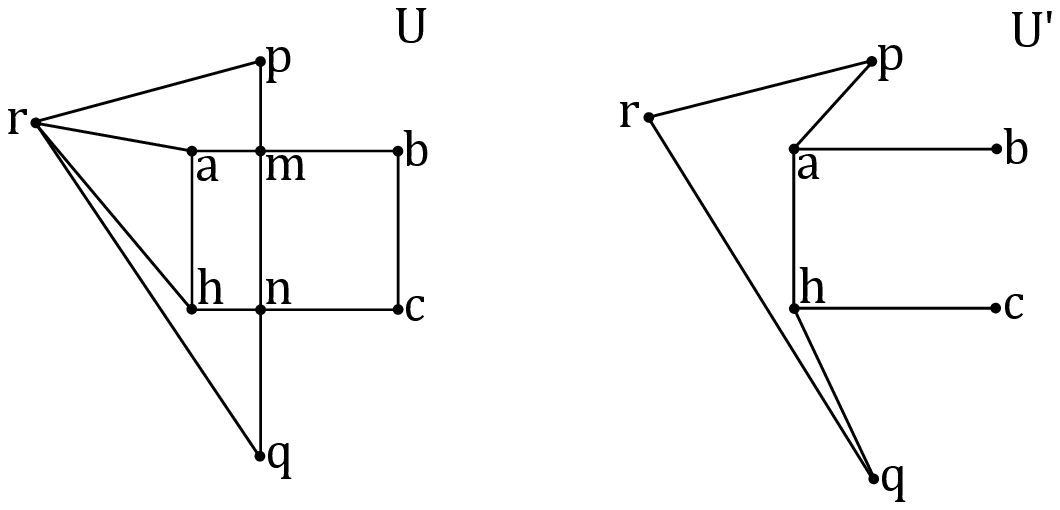}
        \caption{Case B, the quadrilateral \emph{ramp} is convex and the quadrilateral \emph{rhnq} is concave}
 \end{center}
\end{figure}

\begin{figure}[ht]
   \begin{center}
     \includegraphics[height=3.5cm]{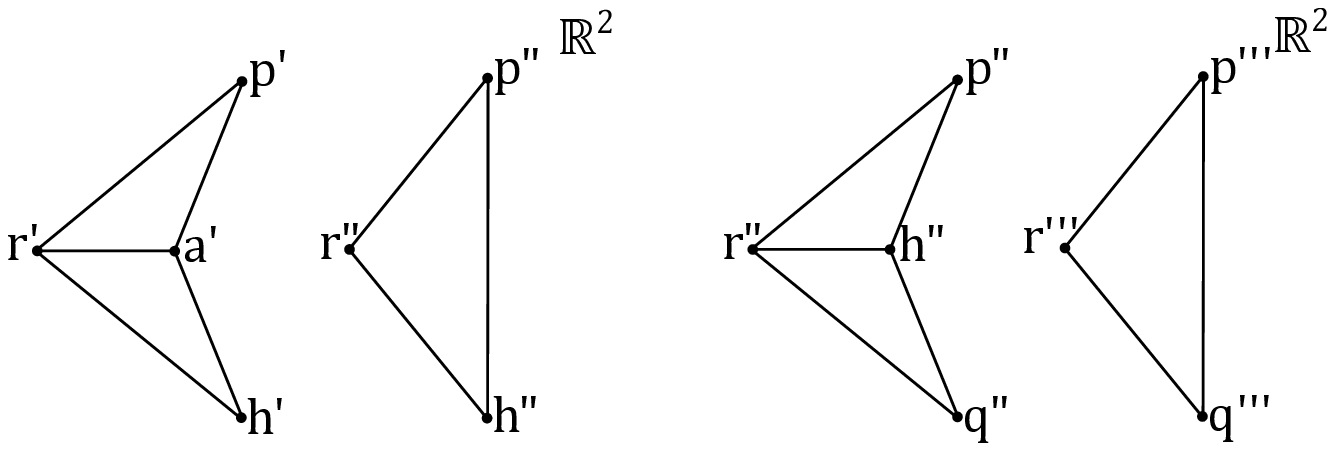}
        \caption{}
 \end{center}
\end{figure}

Let $\triangle(p',a',r')$ be a comparison triangle for $\triangle(p,a,r)$ in $U$. Let $\triangle(r',a',h')$ be a comparison triangle for $\triangle(r,a,h)$ in $U$.
We place the comparison triangles $\triangle(p',a',r')$ and $\triangle(r',a',h')$ in different half planes with respect to the line $r'a'$. The CAT(0) inequality implies that
$\angle_{p}(r,a) \leq \angle_{p'}(r',a')$, $\angle_{r}(p,a) \leq \angle_{r'}(p',a')$, $\angle_{r}(a,h) \leq \angle_{r'}(a',h')$.

Let $\triangle(p'',r'',h'')$ be a comparison triangle for $\triangle(p,r,h)$ in $U$.
Let $\triangle(q'',r'',h'')$ be a comparison triangle for $\triangle(q,r,h)$ in $U$. We place the comparison triangles $\triangle(p'',r'',h'')$ and $\triangle(q'',r'',h'')$ in different half planes with respect to the line $r''h''$.
Note that either $\angle_{a'}(p',r') + \angle_{a'}(r',h') \geq \pi$ or $\angle_{r'}(p',a') + \angle_{r'}(a',h') \geq \pi$. Assume w.l.o.g. $\angle_{a'}(p',r') + \angle_{a'}(r',h') \geq \pi$. Alexandrov's Lemma implies that $\angle_{p'}(a',r') \leq \angle_{p''}(h'',r'')$, $\angle_{r'}(p',a') + \angle_{r'}(a',h') \leq \angle_{r''}(p'',h'')$. Hence $\angle_{r}(p,h) \leq \angle_{r}(p,a) + \angle_{r}(a,h) \leq \angle_{r'}(p',a') + \angle_{r'}(a',h') \leq \angle_{r''}(p'',h'')$.

 The CAT(0) inequality implies that $\angle_{r}(h,q) \leq \angle_{r''}(h'',q'')$, $\angle_{q}(r,h) \leq \angle_{q''}(r'',h'')$, $\angle_{r}(p,q) \leq \angle_{r}(p,h) + \angle_{r}(h,q) \leq \angle_{r''}(p'',h'') + \angle_{r''}(h'',q'') = \angle_{r''}(p'',q'')$.

Let $\triangle(p''',r''',q''')$ be a comparison triangle for $\triangle(p,r,q)$ in $U$.
Note that either $\angle_{h''}(p'',r'') + \angle_{h''}(r'',q') \geq \pi$ or $\angle_{r''}(p'',h'') + \angle_{r''}(h'',q') \geq \pi$. Assume w.l.o.g. $\angle_{h''}(p'',r'') + \angle_{h''}(r'',q') \geq \pi$. Alexandrov's Lemma implies that $\angle_{r''}(p'',h'') + \angle_{r''}(h'',q'') \leq \angle_{r'''}(p''',q''')$, $\angle_{p''}(r'',h'') \leq \angle_{p'''}(r''',q''')$, $\angle_{q''}(r'',h'') \leq \angle_{q'''}(r''',p''')$.

Hence $\angle_{p}(r,q) = \angle_{p}(r,a) \leq \angle_{p'}(r',a') \leq \angle_{p''}(r'',h'') \leq \angle_{p'''}(r''',q''')$, $\angle_{q}(r,h) \leq \angle_{q''}(r'',h'') \leq \angle_{q'''}(r''',p''')$, $\angle_{r}(p,a) + \angle_{r}(a,h) \leq \angle_{r'}(p',a') + \angle_{r'}(a',h') \leq \angle_{r''}(p'',h'')$, $\angle_{r}(p,q) \leq \angle_{r}(p,a) + \angle_{r}(a,h) + \angle_{r}(h,q) \leq \angle_{r''}(p'',h'') + \angle_{r''}(h'',q'') \leq \angle_{r'''}(p''',q''')$.
So the geodesic
triangle $\triangle (p,r,q)$ in $U'$ satisfies the CAT(0)
inequality.

We show further that the geodesic
triangle $\triangle (p,n,r)$ in $U'$ satisfies the CAT(0)
inequality.

\begin{figure}[ht]
   \begin{center}
     \includegraphics[height=2.3cm]{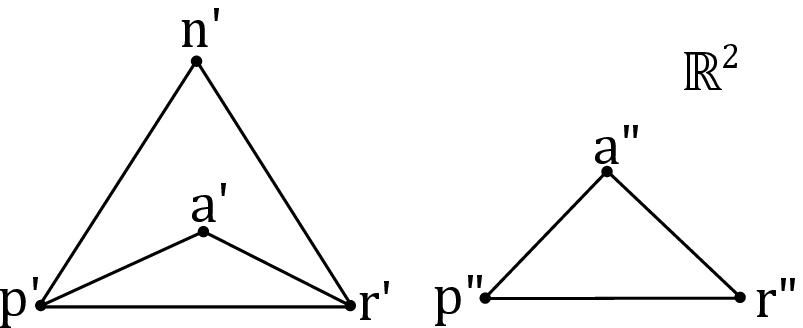}
        \caption{}
 \end{center}
\end{figure}

Let $\triangle(p'',a'',r'')$ be a comparison triangle for $\triangle (p,a,r)$ in $U'$. The CAT(0) inequality implies that $\angle_{p}(a,r) \leq \angle_{p''}(a'',r'')$,
$\angle_{r}(p,a) \leq \angle_{r''}(p'',a'')$.

Let $\triangle(p',n',r')$ be a comparison triangle for $\triangle (p,n,r)$ in $U'$. We consider a point $a'$ in the interior of the geodesic triangle
$\triangle (p',n',r')$ such that $d_{\mathds{R}^{2}}(p',a') =
d(p,a)$  and $d_{\mathds{R}^{2}}(r',a') = d(r,a)$. Thus
$\angle_{p'} (a',r') \leq \angle_{p'} (n',r')$ and $\angle_{r'}
(a',p') \leq \angle_{r'} (n',p')$. Because the geodesic triangles
$\triangle (p',a',r')$ and $\triangle (p'',a'',r'')$ are congruent
to each other, we have $\angle_{p'} (a',r') \equiv \angle_{p''} (a'',r'')$
and $\angle_{r'} (a',p') \equiv \angle_{r''} (a'',p'')$.

In conclusion for the geodesic triangle $\triangle(p,n,r)$ in $U'$ we have $\angle_{p}(n,r) < \angle_{p'}(n',r')$, $\angle_{r}(n,p) < \angle_{r'}(n',p')$. Because the Alexandrov angle between the geodesic segments $[p,n]$ and $[r,n]$ equals zero, we have
$\angle_{n}(r,p) = 0 < \angle_{n'}(p',r')$. Hence the geodesic
triangle $\triangle (p,n,r)$ in $U'$ satisfies the CAT(0)
inequality.

\end{proof}

\begin{lemma}\label{3.21}
Let the segment $[p,q]$ be as in case B. Let $r$ be a point in $U$ such that the segment $[r,q]$ does not intersect $\sigma$ and the segment $[p,r]$ intersects $\sigma$ ($[p,r] \cap [a,b] \neq \emptyset$, $[p,r] \cap [c,h] \neq \emptyset$). Then, the geodesic
triangle $\triangle (p,r,q)$ in $U'$ satisfies the CAT(0) inequality.
\end{lemma}

\begin{proof}

By Lemma \ref{3.11}, $d'(p,q) = d'(p,a) + d'(a,h) + d'(h,q)$ and $d'(p,r) = d'(p,a) + d'(a,h) + d'(h,r)$.

\begin{figure}[ht]
   \begin{center}
     \includegraphics[height=3.5cm]{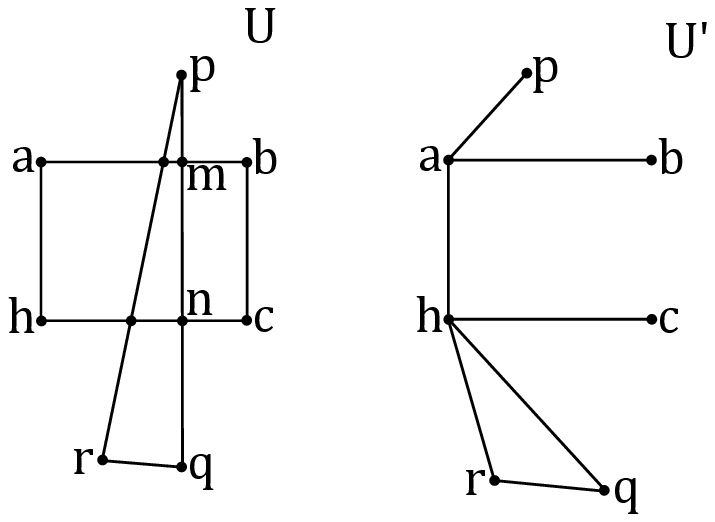}
        \caption{Case B}
 \end{center}
\end{figure}

\begin{figure}[ht]
   \begin{center}
     \includegraphics[height=2.5cm]{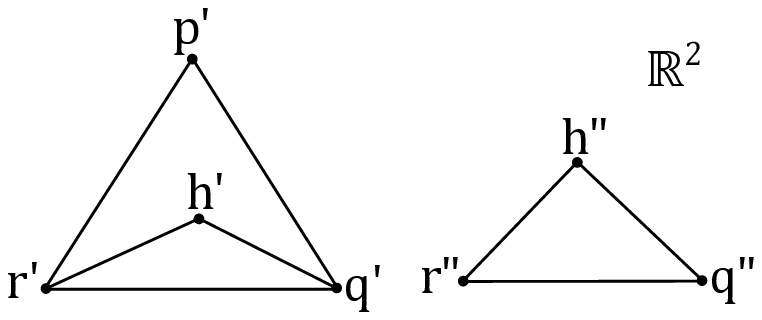}
        \caption{}
 \end{center}
\end{figure}

Let $\triangle(r'', h'', q'')$ be a comparison triangle for $\triangle(r, h, q)$ in $U$. Because $U$ is a CAT(0) space, we have $\angle_{r}(h,q) \leq \angle_{r''}(h'',q'')$,
$\angle_{q}(h,r) \leq \angle_{q''}(h'',r'')$.

Let $\triangle(p', q', r')$ be a comparison triangle for $\triangle(p, q, r)$ in $U$.
We consider a point $h'$ in the interior of the geodesic triangle
$\triangle (p',q',r')$ such that $d_{\mathds{R}^{2}}(r',h') =
d(r,h)$ and $d_{\mathds{R}^{2}}(q',h') = d(q,h)$. Thus
$\angle_{r'} (h',q') \leq \angle_{r'} (p',q')$ and $\angle_{q'}
(h',r') \leq \angle_{q'} (p',r')$. Because the geodesic triangles
$\triangle (r',h',q')$ and $\triangle (r'',h'',q'')$ are congruent
to each other, we have $\angle_{r'} (h',q') \equiv \angle_{r''} (h'',q'')$, $\angle_{q'} (h',r') \equiv \angle_{q''} (h'',r'')$.

In conclusion in $U'$ we have $\angle_{r}(p,q) \leq \angle_{r'}(p',q')$,
$\angle_{q}(p,r) \leq \angle_{q'}(p',r')$,
$\angle_{p}(r,q) = 0 < \angle_{p'}(r',q')$.
Thus the geodesic triangle $\triangle (p,r,q)$ in $U'$ satisfies the CAT(0)
inequality.

\end{proof}

\begin{lemma}\label{3.23}
Let the segment $[p,q]$ be as in case B. Let $r$ be a point in $U$ such that the segments $[p,r]$ and $[r,q]$ intersect $\sigma$ ($[p,r] \cap [a,b] \neq \emptyset$, $[p,r] \cap [a,h] \neq \emptyset$, $[r,q] \cap [a,h] \neq \emptyset$, $[r,q] \cap [c,h] \neq \emptyset$). Then, the geodesic
triangle $\triangle (p,r,q)$ in $U'$ satisfies the CAT(0) inequality.
\end{lemma}

\begin{proof}

By Lemma \ref{3.9}, $d'(p,r) = d'(p,a) + d'(a,r)$ and $d'(r,q) = d'(r,h) + d'(h,q)$.
Lemma \ref{3.11} implies that $d'(p,q) = d'(p,a) + d'(a,h) + d'(h,q)$.

Since $d'(a,h) < d'(a,r) + d'(r,h)$, we have that $d'(p,q) < d'(p,r) + d'(r,q)$. So in $U'$ the geodesic triangle $\triangle (p,q,r)$ is well defined.

Let $\triangle(p',q',r')$ be a comparison triangle for $\triangle(p,q,r)$ in $U'$. Let $\triangle(r'',a'',h'')$ be a comparison triangle for $\triangle(r,a,h)$ in $U$. Let $a' \in [p',r']$ be a comparison point for $a \in [p,r]$. Let $h' \in [r',q']$ be a comparison point for $h \in [r,q]$. Note that in $U'$ we have $0 = \angle_{p}(r,q) < \angle_{p'}(r',q')$ and $0 = \angle_{q}(r,p) < \angle_{q'}(r',p')$.

\begin{figure}[ht]
   \begin{center}
     \includegraphics[height=2.5cm]{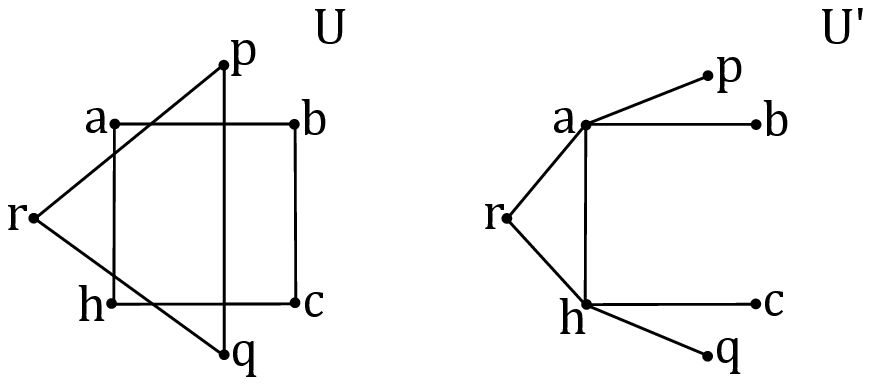}
        \caption{Case B}
 \end{center}
\end{figure}

\begin{figure}[ht]
   \begin{center}
     \includegraphics[height=2.5cm]{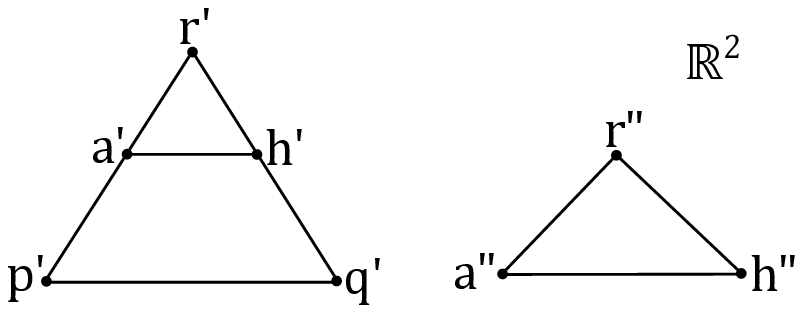}
        \caption{}
 \end{center}
\end{figure}

By the CAT(0) inequality we have that $d(a,h) \leq d_{\mathds{R}^{2}}(a',h')$. Hence, since $d_{\mathds{R}^{2}}(a'',h'') = d(a,h)$, we have $d_{\mathds{R}^{2}}(a'',h'')$ $\leq d_{\mathds{R}^{2}}(a',h')$. Then, because $d_{\mathds{R}^{2}}(a',r') = d_{\mathds{R}^{2}}(a'',r'')$ and $d_{\mathds{R}^{2}}(h',r') = d_{\mathds{R}^{2}}(h'',r'')$, we have $\angle_{r''}(a'',h'') \leq \angle_{r'}(a',h')$.

The CAT(0) inequality implies in $U'$ we have that $\angle_{r}(p,q) = \angle_{r}(a,h) \leq \angle_{r''}(a'',h'') \leq \angle_{r'}(a',h') = \angle_{r'}(p',q')$.

So the geodesic
triangle $\triangle (p,r,q)$ in $U'$ satisfies the CAT(0) inequality.

\end{proof}

\begin{proposition}\label{3.25}
Every point in $|K'|$ has a neighborhood that is a CAT(0) space.
\end{proposition}

\begin{proof}

Let $u,v,w$ be three distinct points in $U$ chosen such that they do
not belong to $\sigma$ and such that the geodesic segments $[u,v]$,
$[u,w]$ and $[v,w]$ in $U$ do not intersect $\sigma$. Note that the
geodesic triangle $\triangle (u,v,w)$ in $U'$ satisfies the CAT(0)
inequality. Hence, due to the Lemmas \ref{3.13} - \ref{3.23}, we may conclude that any geodesic triangle in $U'$
fulfills the CAT(0) inequality. So $U'$ is a CAT(0) space.

Let $y$ be a point in $|K|$ that does not belong to $\sigma$. Let
$U_{y}$ be a neighborhood of $y$ homeomorphic to a closed ball of
radius $r_{y}$, $U_{y} = \{x \in |K| \mid d(y,x) \leq r_{y}\}$. The
radius $r_{y}$ is chosen small enough such that $U_{y}$ does not
intersect $\sigma$. For any $y$ in $|K'|$ that does not lie on
$[a,b], [a,c]$ or $[a,h]$, we consider a neighborhood $U_{y}'$ that
coincides with $U_{y}$. $U'_{y}$ is hence a CAT(0) space.

So every point in $|K'|$ has a neighborhood which is a CAT(0) space.

\end{proof}

We are now in the position to show the main result of the paper: any
finite, CAT(0) square $2$-complex retracts to a point through subspaces
which remain, at each step of the retraction, CAT(0) spaces.

\begin{theorem}\label{3.27}
Let $K$ be a finite, CAT(0) square $2$-complex.
Then $K$ collapses to a point through CAT(0) subspaces $|K'|$.
\end{theorem}

\begin{proof}

Proposition $3.1$ implies that $K$ has a $2$-cell with a free
$1$-dimensional face. We fix a point $p$ in the interior of a
$2$-cell of $K$. We define the map $R : |K| \times [0,1]
\rightarrow |K|$ which associates for any $x \in |K|$ and for any $t
\in [0,1]$, to $(x,t)$ the point a distance $t \cdot d(p,x)$ from
$x$ along the geodesic segment $[p,x]$. Because $|K|$ has a strongly
convex metric, the map $R$ is a continuous retraction of $|K|$ to
$p$. $R(|K| \times [0,1])$ is therefore contractible and then it is simply
connected. Let $a,b,c$ be any three distinct points in $R(|K| \times
[0,1])$ such that the unique geodesic segment $[b,c]$ belongs to a
$1$-cell that is the face of a single $2$-cell $\sigma$ in the complex.
Also the points $a,b,c$ are chosen such that the geodesic triangle $\delta = \triangle (a,b,c)$
is contained in $\sigma$.
For each such $\delta = \triangle (a,b,c)$, we deformation retract $R(|K|
\times [0,1])$ by pushing in $\delta$ starting at $[b,c]$. We obtain
each time a subspace $ |K'| = R(|K| \times [0,1])$ which remains
simply connected and, by Proposition \ref{3.25}, non-positively
curved. So $|K'|$ is a CAT(0) space implying that any two points in
$|K'|$ are joined by a unique geodesic segment in $|K'|$. If at a
certain step we delete the point $p$, we fix another point $p$ in the
interior of a $2$-cell of $K'$, define the map $R$ as before and
retract the space through CAT(0) subspaces further. Since $K$ is finite,
we reach, after a finite number of steps, a $1$-dimensional spine
$L$. Since $|L|$ is also a CAT(0) space, it is contractible. Taking
into account that a contractible $1$-complex is collapsible, the
result follows.

\end{proof}

As a consequence of the above result we have the following.

\begin{corollary}
Let $K$ be a locally finite, CAT(0) square $2$-complex.
Then $K$ has an arborescent structure.
\end{corollary}

\begin{proof}
We fix a vertex $v$ of $K$. For each integer $n$, let $B_{n}$ be the full subcomplex of $K$ generated by the vertices that can be joined to $v$ by an edge-path of length at most $n$. Note that for each $n$, $B_{n}$ is a ball in a CAT(0) space. It is therefore contractible and, in particular, it is simply connected.

Furthermore, because $K$ is a CAT(0) space, it is locally a CAT(0) space. For each $n > 0$, $B_{n}$ is therefore itself locally a CAT(0) space. Since each $B_{n}, n > 0$ is simply connected and locally a CAT(0) space, it is a CAT(0) space. Hence for each $n$, $B_{n}$ is a finite square $2$-complex that is a CAT(0) space. So, according to Theorem \ref{3.27}, each such $B_{n}$ is collapsible.
Therefore $K$ is the monotone union $\cup_{n=1}^{\infty}B_{n}$ of a sequence of collapsible subcomplexes. This ensures that $K$ has an arborescent structure.

\end{proof}

We give below a few applications of the collapsibility of square $2$-complexes.

Due to the fact that median $2$-complexes are CAT(0) square $2$-complexes (see \cite{chepoi_2000}, Theorem $6.1$), the following holds.

\begin{corollary}
Median $2$-complexes are collapsible.
\end{corollary}

The collapsibility of CAT(0) square $2$-complexes implies that the first derived subdivision of a such complexes is also collapsible (\cite{welker_1999}, Theorem $2.10$). This happens because if
a complex is non-evasive (see \cite{welker_1999}), then it is collapsible (\cite{kahn_1984}, Proposition $1$). Hence the following result holds.

\begin{corollary}
CAT(0) square $2$-complexes admit collapsible triangulations.
\end{corollary}


\begin{thebibliography}{11}

\bibitem{benedetti_2019}
K. A. Adiprasito, B. Benedetti, {\it Collapsibility of CAT(0) spaces}, arXiv:$1107.5789v8$, $2019$.

\bibitem{alex_1955}
A.-D. Alexandrov, {\it Die innere Geometrie der konvexen Flaechen}, Akademie Verlag, Berlin, $1955$.

\bibitem{BCCGO}
B. Bre{\v{s}}ar, J. Chalopin, V. Chepoi, T. Gologranc, D. Osajda, {\it Bucolic complexes}, Adv. Math., $243$, $2013$, $127-167$.

\bibitem{bridson_1999}
M. Bridson, A. Haefliger, {\it Metric spaces of non-positive
curvature}, Springer, New York, $1999$.

\bibitem{burago_2001}
D. Burago, Y. Burago, S. Ivanov, {\it A Course in Metric Geometry}, American Mathematical Society,
Providence, Rhode Island, $2001$.

\bibitem{ChaCHO}
J. Chalopin, V. Chepoi, H. Hirai, D. Osajda, {\it Weakly modular graphs and nonpositive curvature}, Mem. Amer. Math. Soc., to appear.

\bibitem{chepoi_2000}
V. Chepoi, {\it Graphs of some CAT(0) complexes}, Adv. in Appl. Math. $24$, $2$, $2000$, $125-179$.

\bibitem{chepoi_2009}
V. Chepoi, D. Osajda, {\it Dismantlability of weakly systolic complexes and applications}, Trans. Amer. Math. Soc. $367$, $2$, $2015$, $1247 - 1272$.

\bibitem{corson_1998}
J.-M. Corson, B. Trace, {\it The $6$-property for simplicial
complexes and a combinatorial Cartan-Hadamard theorem for
manifolds}, Proceedings of the American Mathematical Society $126$, $3$, $1998$, $917-924$.

\bibitem{crisp_2002}
J. Crisp, {\it On the CAT($0$) dimension of $2$-dimensional Bestvina-Brady groups}, Algebraic and Geometric Topology $2$, $2002$, $921-936$.

\bibitem{crowley_2008}
K. Crowley, {\it Discrete Morse theory and the geometry of
nonpositively curved simplicial complexes}, Geometriae Dedicata $133$, $1$, $2008$, $35 - 50$.

\bibitem{forman_2002}
R. Forman, {\it A user's guide to discrete Morse theory}, Sem. Lothar. Combin. $48$, $2002$, Art. B$48$c, $35$.

\bibitem{forman_1998}
R. Forman, {\it Morse theory for cell complexes}, Adv. Math.
$134$, $1$, $1998$, $90 - 145$.

\bibitem{Hag}
F. Haglund, {\it Complexes simpliciaux hyperboliques de grande dimension}, preprint, Prepublication Orsay,
$71$, $2003$, avalable at http://www.math.u-psud.fr/~haglund/cpl\_hyp\_gde\_dim.pdf

\bibitem{JS0}
T. Januszkiewicz, J. {\'S}wi{\c{a}}tkowski, {\it Hyperbolic Coxeter groups of large dimension},
Comment. Math. Helv, $78$, $3$, $2003$, $555-583$.

\bibitem{JS1}
T. Januszkiewicz, J. {\'S}wi{\c{a}}tkowski, {\it Simplicial nonpositive curvature},
Publ. Math. IHES, $2006$, $1 - 85$.

\bibitem{JS2}
T. Januszkiewicz, J. {\'S}wi{\c{a}}tkowski, {\it Filling invariants of systolic complexes and groups},
Geom. Topol., $11$, $2007$, $727-758$.

\bibitem{kahn_1984}
J. Kahn, M. Saks, D. Sturtevant, {\it A topological approach to evasiveness}, Combinatorica $4$, $1984$, $297 - 306$.

\bibitem{lazar_2010_8}
I.-C. Laz\u{a}r, {\it The study of simplicial complexes of nonpositive curvature}, Ph.D. thesis,
Cluj University Press, $2010$ (http://www.ioana-lazar.ro/phd.html).

\bibitem{lazar_2012}
I.-C. Laz\u{a}r, {\it Discrete Morse theory, simplicial nonpositive curvature, and simplicial collapsibility}, Balkan Jour. Geom. Appl. $17$, $1$ $2012$, $58 - 69$.

\bibitem{lazar_2013}
I.-C. Laz\u{a}r, {\it Systolic simplicial complexes are collapsible}, Bull. Math. Soc. Sci. Math. Roumanie
Tome $56(104)$, $2$, $2013$, $229-236$.

\bibitem{L-8loc}
I.-C. Laz\u{a}r, {\it A combinatorial negative curvature condition implying Gromov hyperbolicity}, $2015$, available at https://arxiv.org/pdf/1501.05487.pdf.

\bibitem{lazar_2015}
I.-C. Laz\u{a}r, {\it Minimal disc diagrams of $5/9$-simplicial complexes}, to appear in Michigan Math. J., $2019$, available at arXiv:$1509.03760$.

\bibitem{menger_1928}
K. Menger, {\it Untersuchungen ueber allgemeine Metrik}, Math. Ann., $100$, $1928$, $75-163$.

\bibitem{O-conn}
D. Osajda, {\it Connectedness at infinity of systolic complexes and groups}, Groups Geom. Dyn., $1$, $2$,
$2007$, $183-203$.

\bibitem{O-ib}
D. Osajda, {\it Ideal boundary of 7-systolic complexes and groups}, Algebr. Geom. Topol., $8$, $1$,
$2008$, $81-99$.

\bibitem{O-chcg}
D. Osajda, {\it A construction of hyperbolic Coxeter groups}, Comment. Math. Helv., $88$, $2$,
$2013$, $353-367$.

\bibitem{O-sdn}
D. Osajda, {\it A combinatorial non-positive curvature I: weak systolicity}, preprint,
$2013$, available at arXiv:$1305.4661$

\bibitem{O-8loc}
D. Osajda, {\it Combinatorial negative curvature and triangulations of three-manifolds}, Indiana Univ. Math. J., $64$, $3$,
$2015$, $943-956$.

\bibitem{O-ns}
D. Osajda, {\it Normal subgroups of SimpHAtic groups}, submitted, $2015$, available at arXiv:$1501.00951$

\bibitem{OS}
D. Osajda, J. {\'S}wi{\c{a}}tkowski, {\it On asymptotically hereditarily aspherical groups}, Proc. London Math. Soc., $111$, $1$,
$2015$, $93-126$.

\bibitem{welker_1999}
V. Welker, {\it Constructions preserving evasiveness and collapsibility}, Discrete Math., $207$,
$1999$, $243 - 255$.

\bibitem{white_1967}
W. White, {\it A 2-complex is collapsible if and only if it admits a strongly convex metric}, Notices Amer. Math. Soc., $14$, $1967$, $24 - 28$.

\bibitem{zeeman_1964}
C. E. Zeeman, {\it On the Dunce Hat}, Topology, $2$, $1964$, $341-358$.

\end{thebibliography}
\end{document}